\documentclass[12pt,reqno,letterpaper]{amsart}
\usepackage[T1]{fontenc}
\usepackage[utf8]{inputenc}
\usepackage{amsmath}
\usepackage{amssymb}
\usepackage{amsthm} 
\usepackage{amstext} 
\usepackage[margin=1in, footskip=0.3in]{geometry}
\usepackage{hyperref}
\usepackage{cite}
\allowdisplaybreaks[4]
\hypersetup{
colorlinks=true,
linkcolor=blue,
urlcolor=red,
citecolor=green,
}

\newtheorem{theorem}{Theorem}[section]
\newtheorem{lemma}[theorem]{Lemma}
\newtheorem{conjecture}{Conjecture}
\numberwithin{equation}{section}

\begin{document}

\title{Endpoint absolute monotonicity for the complete elliptic integral of the first kind}

\author[Deng]{Qintao Deng}
\address{[Qintao Deng] School of Mathematics and Statistics \& Hubei Key Laboratory of Mathematical Sciences, Central China Normal University, Wuhan 430079, P.R. China}
\email{qintaodeng@ccnu.edu.cn}

\author[Kou]{Yunjia Kou$^{\dagger}$}
\address{[Yunjia Kou] Department of Mathematics and Statistics, Washington State University, Pullman, WA 99164-3113, USA}
\email{yunjia.kou@wsu.edu}

\author[Zhong]{Fulin Zhong}
\address{[Fulin Zhong] School of Mathematics and Statistics, Central China Normal University, Wuhan 430079, P. R. China}
\email{flzhong@mails.ccnu.edu.cn}

\thanks{$^{\dagger}$ Corresponding author: Yunjia Kou}


\subjclass[2020]{33E05, 
26A48, 
26D15 
}
\keywords{complete elliptic integral of the first kind, absolute monotonicity, Maclaurin coefficients}

\begin{abstract} 
Let
\begin{equation*}
f_p(x)=\frac{\mathcal{K}(\sqrt{x})}{p-\ln\sqrt{1-x}},
\end{equation*}
where $\mathcal{K}$ denotes the complete elliptic integral of the first kind, and set $g_p=1/f_p$. Tian and Yang conjectured that, at $p=\ln4$, both $-f_{\ln4}^{\prime\prime\prime}$ and $g_{\ln4}$ are absolutely monotonic on $(0,1)$. Writing
\begin{equation*}
\frac{2}{\pi}f_{\ln4}(x)=\sum_{n=0}^{\infty}a_n(\ln4)x^n,\quad \frac{\pi}{2}g_{\ln4}(x)=\sum_{n=0}^{\infty}b_n(\ln4)x^n,
\end{equation*}
we prove
\begin{equation*}
a_n(\ln4)<0\quad(n\geq3),\quad b_n(\ln4)>0\quad(n\geq0),
\end{equation*}
thereby settling both conjectures. The proof of the first inequality is based on the representation
\begin{equation*}
a_n(\ln4)=\frac{(-1)^nC}{15^{n+1}}-M_n,
\end{equation*}
where $C>0$ and $(M_n)$ is a positive moment sequence. Its strict log-convexity, together with estimates for the initial coefficients, determines the sign of every $a_n$. The second inequality follows from a factorization of the quadratic truncation of the normalized Taylor series and a reciprocal-series argument. 
\end{abstract}

\maketitle

\section{Introduction}

\subsection{Background and the conjectures}

For $0<r<1$, Legendre's complete elliptic integral of the first kind is defined by
\begin{equation*}
\mathcal{K}(r)=\int_0^{\frac{\pi}{2}}\frac{\mathrm{d}t}{\sqrt{1-r^2\sin^2t}}
=\frac{\pi}{2}{}_2F_1\left(\frac{1}{2},\frac{1}{2};1;r^2\right).
\end{equation*}
As $r\to1^-$, $\mathcal{K}(r)$ has the logarithmic behavior
\begin{equation*}
\mathcal{K}(r)=\ln\frac{4}{r'}+O\left((r')^2\ln\frac{1}{r'}\right), \quad r'=\sqrt{1-r^2}.
\end{equation*}
This estimate follows, for example, from \cite{PonnusamyVuorinen1997}. It naturally leads to the comparison of $\mathcal{K}(r)$ with logarithmic functions and to the study of how monotonicity and convexity depend on the constant appearing in the logarithmic term.

Complete elliptic integrals have been extensively studied from the viewpoints of sharp bounds, mean inequalities, monotonicity and concavity \cite{AVV1990,AVV1992,QiuVamanamurthy1996,QiuVamanamurthyVuorinen1998,AlzerQiu2004,ChuQiuWang2012,YangQianChu2018,AlzerRichards2020,RichardsSmith2021}. Logarithmic-type estimates closely related to the asymptotic behavior above were obtained in \cite{Wang2025}. 
Through the hypergeometric representation of $\mathcal{K}$, these questions are also connected with zero-balanced and generalized hypergeometric functions. Bounds and ratio inequalities for such functions were studied in \cite{AndersonBarnard1995,QiuVuorinen2000,KarpSitnik2009,Richards2019,BarnardRichardsSliheet2020}. Further results include Tur\'an-type inequalities for generalized complete elliptic integrals \cite{Baricz2007} and parameter log-convexity and log-concavity for hypergeometric-type functions \cite{KarpSitnik2010}.

Absolute monotonicity for functions associated with complete elliptic integrals and zero-balanced hypergeometric functions has been investigated in \cite{YangTian2021,SunWangHuang2025,WangYang2026,YangZhao2026}. Recurrence relations for Maclaurin coefficients were developed in \cite{Yang2025}, while further criteria for absolute monotonicity in Gaussian hypergeometric families were obtained in \cite{Zhao2026}.

The present paper concerns two endpoint conjectures for a quotient involving the complete elliptic integral and its reciprocal. For $p>0$, let
\begin{equation}\label{fpdef}
f_p(x)=\frac{\mathcal{K}(\sqrt{x})}{p-\ln\sqrt{1-x}},\quad 0<x<1,
\end{equation}
and for $q\geq0$, let
\begin{equation}\label{gqdef}
g_q(x)=\frac{q-\ln\sqrt{1-x}}{\mathcal{K}(\sqrt{x})},\quad 0<x<1.
\end{equation}
Thus $g_q=1/f_q$ for $q>0$. Since
\begin{equation*}
p-\ln\sqrt{1-x}=\ln\frac{\mathrm{e}^p}{\sqrt{1-x}},
\end{equation*}
the choice $p=\ln4$ matches the constant in the logarithmic asymptotic expansion of $\mathcal{K}(\sqrt{x})$ as $x\to1^-$. In particular, $f_p(x)\to1$ as $x\to1^-$ for every $p>0$. 
Yang and Tian \cite{YangTian2019Convexity} proved that $f_p$ is strictly concave on $(0,1)$ if and only if $p=\frac{4}{3}$.  

Recall that a function $h$ is absolutely monotonic on an interval $I$ if it has derivatives of all orders and
\begin{equation*}
h^{(m)}(x)\geq0,\quad x\in I,\  m\geq0.
\end{equation*}
Tian and Yang \cite{TianYang2022} proved that $-f_p''$ is absolutely monotonic on $(0,1)$ if and only if $p=\frac{4}{3}$, and that $f_p'''$ is absolutely monotonic if and only if $p\geq p_0$, where $p_0=1.506583\ldots$ is the unique real zero of
\begin{equation*}
75p^3-230p^2+240p-96.
\end{equation*}
They also proved that $f_p'$ and $f_p$ are absolutely monotonic if and only if $p\geq2$. 
For the reciprocal family, Tian and Yang \cite{TianYang2022} obtained the necessary and sufficient conditions for the absolute monotonicity of $-g_q'''$, $-g_q''$ and $-g_q'$ on $(0,1)$, namely,
\begin{equation*}
q\geq\frac{50}{33},\quad q\geq\frac{8}{5},\quad q\geq2,
\end{equation*}
respectively. They also proved that $g_q$ is absolutely monotonic on $(0,1)$ for $0\leq q\leq1$.

The logarithmic parameter $\ln4$ is not covered by these results. Indeed,
\begin{equation*}
\frac{4}{3}<\ln4<p_0,\quad 1<\ln4<\frac{50}{33}.
\end{equation*}
Thus the preceding criteria do not determine whether $-f_{\ln4}'''$ and $g_{\ln4}$ are absolutely monotonic. Tian and Yang formulated the following two endpoint conjectures in \cite{TianYang2022}.
 
\begin{conjecture}
Let $f_p$ be defined by \eqref{fpdef}. Then $-f_{\ln4}'''$ is absolutely monotonic on $(0,1)$.
\end{conjecture}

\begin{conjecture}
Let $g_q$ be defined by \eqref{gqdef}. Then $g_{\ln4}$ is absolutely monotonic on $(0,1)$.
\end{conjecture}

The purpose of the present paper is to prove these two conjectures.

\subsection{Main results}

Let
\begin{equation*}
W_n=\frac{\Gamma(n+\frac{1}{2})}{\Gamma(\frac{1}{2})\Gamma(n+1)}.
\end{equation*}
Following the normalization in \cite{TianYang2022}, write
\begin{equation}\label{originalseries}
\mathcal{K}(\sqrt{x})=\frac{\pi}{2}\sum_{n=0}^{\infty}W_n^2x^n,\quad f_p(x)=\frac{\pi}{2}\sum_{n=0}^{\infty}a_n(p)x^n,\quad g_q(x)=\frac{2}{\pi}\sum_{n=0}^{\infty}b_n(q)x^n.
\end{equation}
The first series converges for $|x|<1$. The latter two are initially understood as Taylor expansions in a neighborhood of the origin. At $p=q=\ln4$, their analyticity in the unit disk will be established in Sections \ref{sec-first} and \ref{sec-second}. 
 
\begin{theorem}\label{main}
The coefficients in \eqref{originalseries} satisfy
\begin{equation*}
a_n(\ln4)<0\ \ (n\geq3),\quad b_n(\ln4)>0\ \ (n\geq0).
\end{equation*}
Consequently,
\begin{equation*}
\left(-f_{\ln4}'''\right)^{(m)}(x)>0,\quad g_{\ln4}^{(m)}(x)>0
\end{equation*}
for every integer $m\geq0$ and every $x\in(0,1)$. In particular, Conjectures 1 and 2 hold.
\end{theorem}

The main difficulty lies in determining the signs of all higher coefficients at $p=\ln4$. The recurrence for $a_n(p)$ involves coefficients of both signs, and a negative coefficient contributes positively to the next recurrence step. Thus the recurrence is not sign preserving, and a direct induction does not close.

We overcome this difficulty by meromorphic continuation and contour integration. The logarithmic denominator extends to $\mathbb{C}\setminus[1,\infty)$ and has a unique zero at $-15$, producing a simple pole of the normalized generating function. The residue theorem then gives
\begin{equation} \label{introformula}
a_n(\ln4)=\frac{(-1)^nC}{15^{n+1}}-M_n,\quad
M_n=\int_1^\infty\frac{\rho(t)}{t^{n+1}}\mathrm{d}t,
\end{equation}
where $C>0$ and $\rho(t)>0$ for $t>1$. The positivity of $\rho$ follows from
\begin{equation*}
\frac{\pi}{2}\frac{\mathcal{K}(r)}{\mathcal{K}(r')}>\ln\frac{4r}{r'},\quad 0<r<1.
\end{equation*}
The moment sequence $(M_n)$ is strictly log-convex, and hence the ratios $M_{n+1}/M_n$ are strictly increasing. The estimates for $a_2(\ln4)$ and $a_3(\ln4)$ give
\begin{equation*}
\frac{M_3}{M_2}>\frac{1}{15}.
\end{equation*}
Since the ratios $M_{n+1}/M_n$ are strictly increasing, it follows that
\begin{equation*}
M_n>\frac{C}{15^{n+1}},\quad n\geq3.
\end{equation*}
Together with \eqref{introformula}, this yields
\begin{equation*}
a_n(\ln4)<0,\quad n\geq3.
\end{equation*}

For the second inequality, write
\begin{equation*}
\frac{2}{\pi}f_{\ln4}(x)=P(x)-H(x),
\end{equation*}
where
\begin{equation*}
P(x)=a_0+a_1x+a_2x^2,\quad
H(x)=\sum_{n=3}^{\infty}(-a_n)x^n.
\end{equation*}
The coefficients of $H$ are strictly positive, while
\begin{equation*}
P(x)=\frac{1}{\ln4}(1-\lambda x)(1-\mu x),\quad \lambda,\mu>0.
\end{equation*}
Hence $1/P$ has strictly positive Taylor coefficients. Since
\begin{equation*}
\frac{1}{P-H}=\frac{1}{P}\frac{1}{1-H/P},
\end{equation*}
the geometric-series expansion of the second factor gives
\begin{equation*}
b_n(\ln4)>0,\quad n\geq0.
\end{equation*}
Finally, $A$ has no zeros in the unit disk, so $1/\varphi=D/A$ is analytic there and its Taylor series converges throughout $|z|<1$.

The paper is organized as follows. Section \ref{sec-first} proves the first coefficient assertion in Theorem \ref{main} and hence Conjecture 1. Section \ref{sec-second} proves the reciprocal coefficient assertion and hence Conjecture 2.

\section{Proof of Conjecture 1} \label{sec-first}
In this section, we first derive the algebraic recurrence to explicitly compute the initial coefficients. To govern the signs of all higher-order coefficients at $p=\ln 4$, we analytically continue the generating function to the complex plane and use contour integration to extract the moment representation \eqref{introformula}. The strict positivity of this moment sequence will then be established via real-variable elliptic inequalities and Cauchy-Schwarz propagation.

\subsection{Initial coefficients}

For $|x|<1$,
\begin{equation} \label{denseries} 
p-\ln\sqrt{1-x}=p+\sum_{n=1}^{\infty}\frac{x^n}{2n}.
\end{equation}

\begin{lemma}
For $p>0$, the coefficients $a_n(p)$ in \eqref{originalseries} satisfy $a_0(p)=\frac{1}{p}$ and
\begin{equation} \label{arecurrenceoriginal} 
a_n(p)=\frac{1}{p}W_n^2-\frac{1}{2p}\sum_{k=0}^{n-1}\frac{a_k(p)}{n-k}, \quad n\geq 1.
\end{equation}
In particular,
\begin{equation} \label{coeffs} 
a_0(p)=\frac{1}{p}, \quad a_1(p)=\frac{p-2}{4p^2}, \quad a_2(p)=\frac{(3p-4)^2}{64p^3},
\end{equation}
and
\begin{equation} \label{a3formula} 
a_3(p)=\frac{75p^3-230p^2+240p-96}{768p^4}.
\end{equation}
\end{lemma}

\begin{proof}
For $|x|$ sufficiently small, \eqref{originalseries} and \eqref{denseries} give
\begin{equation*}
\left(p+\sum_{j=1}^{\infty}\frac{x^j}{2j}\right) \left(\sum_{k=0}^{\infty}a_k(p)x^k\right) =\sum_{n=0}^{\infty}W_n^2x^n.
\end{equation*}
On a sufficiently small disk, all three Taylor series converge absolutely, so the product on the left is given by the Cauchy product. The constant coefficient is $pa_0$, while the coefficient of $x^n$, for $n\geq 1$, is
\begin{equation*}
pa_n+\sum_{j=1}^{n}\frac{a_{n-j}}{2j}.
\end{equation*}
Changing the index by $k=n-j$ gives
\begin{equation*}
\sum_{j=1}^{n}\frac{a_{n-j}}{2j} =\frac{1}{2}\sum_{k=0}^{n-1}\frac{a_k}{n-k}.
\end{equation*}
Comparison with the right-hand side therefore yields
\begin{equation*}
pa_0=1
\end{equation*}
and, for every $n\geq 1$,
\begin{equation*}
pa_n+\frac{1}{2}\sum_{k=0}^{n-1}\frac{a_k}{n-k}=W_n^2.
\end{equation*}
Since $p>0$, division by $p$ gives $a_0=1/p$ and \eqref{arecurrenceoriginal}. 

Since
\begin{equation*}
W_0=1,\quad W_n=\frac{2n-1}{2n}W_{n-1}\quad(n\geq1),
\end{equation*}
we obtain
\begin{equation*}
W_1=\frac{1}{2},\quad W_2=\frac{3}{8},\quad W_3=\frac{5}{16}.
\end{equation*}
For $n=1$, formula \eqref{arecurrenceoriginal} gives
\begin{equation*}
a_1=\frac{1}{4p}-\frac{a_0}{2p} =\frac{1}{4p}-\frac{1}{2p^2} =\frac{p-2}{4p^2}.
\end{equation*}
For $n=2$,
\begin{equation*}
a_2=\frac9{64p}-\frac{1}{2p}\left(\frac{a_0}{2}+a_1\right).
\end{equation*}
Using the formulas for $a_0$ and $a_1$,
\begin{equation*}
\frac{a_0}{2}+a_1 =\frac{1}{2p}+\frac{p-2}{4p^2} =\frac{3p-2}{4p^2}.
\end{equation*}
Then
\begin{equation*}
a_2=\frac9{64p}-\frac{3p-2}{8p^3} =\frac{9p^2-24p+16}{64p^3} =\frac{(3p-4)^2}{64p^3}.
\end{equation*}
For $n=3$,
\begin{equation*}
a_3=\frac{25}{256p}-\frac{1}{2p}\left(\frac{a_0}{3}+\frac{a_1}{2}+a_2\right).
\end{equation*}
Using the formulas for $a_0$, $a_1$ and $a_2$,
\begin{equation*}
\frac{a_0}{3}+\frac{a_1}{2}+a_2=\frac{1}{3p}+\frac{p-2}{8p^2}+\frac{(3p-4)^2}{64p^3}=\frac{115p^2-120p+48}{192p^3}.
\end{equation*}
Therefore,
\begin{equation*}
a_3=\frac{25}{256p}-\frac{115p^2-120p+48}{384p^4}=\frac{75p^3-230p^2+240p-96}{768p^4}.
\end{equation*}
This proves the stated formulas.
\end{proof}

\subsection{Meromorphic continuation and boundary values}

From now on fix $p=\ln4$ and write $a_n=a_n(\ln4)$. For $0<x<1$, the definition of $\mathcal{K}$ gives
\begin{equation*}
\frac{2}{\pi}\mathcal{K}(\sqrt{x})=\frac{2}{\pi}\int_0^{\frac{\pi}{2}}\frac{\mathrm{d}\theta}{\sqrt{1-x\sin^2\theta}}.
\end{equation*}
Setting $t=\sin^2\theta$, we have
\begin{equation*}
\mathrm{d}\theta=\frac{\mathrm{d}t}{2\sqrt{t(1-t)}},
\end{equation*}
and hence
\begin{equation*}
\frac{2}{\pi}\mathcal{K}(\sqrt{x})=\frac{1}{\pi}\int_0^1\frac{(1-xt)^{-\frac{1}{2}}}{\sqrt{t(1-t)}}\mathrm{d}t.
\end{equation*}
Set $\Omega=\mathbb{C}\setminus[1,\infty)$ and define
\begin{equation} \label{EulerBeta}
A(z)=\frac{1}{\pi}\int_0^1\frac{(1-zt)^{-\frac{1}{2}}}{\sqrt{t(1-t)}}\mathrm{d}t,\quad z\in\Omega,
\end{equation}
where the principal branch of the inverse square root is used. Thus
\begin{equation*}
A(x)=\frac{2}{\pi}\mathcal{K}(\sqrt{x}),\quad 0<x<1.
\end{equation*}

\begin{lemma}
The function $A$ is analytic on $\Omega$.  
\end{lemma}

\begin{proof}
If $z\in\Omega$ and $0<t\leq 1$, then $1-zt\notin(-\infty,0]$. For $t=0$, we have $1-zt=1$. Let $S\Subset\Omega$ be compact and put
\begin{equation*}
K_S=\{1-zt:z\in S, \ 0\leq t\leq 1\}.
\end{equation*}
Then $K_S$ is compact and disjoint from $(-\infty,0]$, so
\begin{equation*}
d_S:=\operatorname{dist}\left(K_S,(-\infty,0]\right)>0.
\end{equation*}
We use the rising factorial notation
\begin{equation*}
(a)_0:=1,\quad (a)_k:=\prod_{j=0}^{k-1}(a+j), \  k=1,2,\ldots.
\end{equation*}
Then for every integer $m\geq 0$,
\begin{equation*}
\frac{\partial^m}{\partial z^m}(1-zt)^{-\frac{1}{2}} =\left(\frac{1}{2}\right)_m t^m(1-zt)^{-m-\frac{1}{2}},
\end{equation*}
and therefore, for $z\in S$ and $0<t<1$,
\begin{equation*}
\left| \frac{\partial^m}{\partial z^m} \left(\frac{(1-zt)^{-\frac{1}{2}}}{\pi\sqrt{t(1-t)}}\right) \right| \leq \frac{\left(\frac{1}{2}\right)_m d_S^{-m-\frac{1}{2}}}{\pi\sqrt{t(1-t)}}.
\end{equation*}
The right-hand side is integrable on $(0,1)$ and independent of $z\in S$. Differentiation under the integral sign gives
\begin{equation*}
A^{(m)}(z)=\frac{\left(\frac{1}{2}\right)_m}{\pi}\int_0^1 \frac{t^m(1-zt)^{-m-\frac{1}{2}}}{\sqrt{t(1-t)}} \mathrm{d}t, \quad z\in S,
\end{equation*}
so $A$ is analytic on $\Omega$. 
\end{proof}

Define
\begin{equation*}
D(z)=\ln4-\frac{1}{2}\operatorname{Log}(1-z),\quad \varphi(z)=\frac{A(z)}{D(z)},\quad z\in\Omega,
\end{equation*}
where $\operatorname{Log}$ denotes the principal logarithm.
Since the principal logarithm is analytic on $\mathbb{C}\setminus(-\infty,0]$ and $1-z\notin(-\infty,0]$ for $z\in\Omega$, the function $D$ is analytic on $\Omega$. Hence $\varphi$ is meromorphic on $\Omega$, with possible poles only at zeros of $D$.

\begin{lemma} \label{polelemma}
The function $\varphi$ has a simple pole at $-15$. Denote its residue by
\begin{equation} \label{Cdef}
C:=\operatorname*{Res}_{z=-15}\varphi(z)=32A(-15),
\end{equation}
then
\begin{equation} \label{Cbound}
0<C<32.
\end{equation}
\end{lemma}

\begin{proof}
The equation $D(z)=0$ is
\begin{equation*}
\operatorname{Log}(1-z)=\ln 16.
\end{equation*}
Taking the exponential of the identity gives $1-z=16$, because $\exp(\operatorname{Log}(1-z))=1-z$ on $\Omega$. Thus every zero must be $z=-15$. Conversely,
\begin{equation*}
D(-15)=\ln4-\frac{1}{2}\operatorname{Log}(16)=0.
\end{equation*}
Hence $-15$ is the unique zero. Moreover,
\begin{equation*}
D'(z)=\frac{1}{2(1-z)},\quad D'(-15)=\frac{1}{32}.
\end{equation*}
Euler's integral representation gives
\begin{equation*}
0<A(-15)=\frac{2}{\pi}\int_0^{\frac{\pi}{2}}\frac{\mathrm{d}\theta}{\sqrt{1+15\sin^2\theta}}<\frac{2}{\pi}\int_0^{\frac{\pi}{2}}1 \mathrm{d}\theta=1.
\end{equation*} 
Hence $\varphi=A/D$ has a genuine simple pole at $-15$, with
\begin{equation*}
\operatorname*{Res}_{z=-15}\varphi(z)=\frac{A(-15)}{D'(-15)}=32A(-15).
\end{equation*}
This proves \eqref{Cdef}, and the bound $0<A(-15)<1$ gives \eqref{Cbound}.
\end{proof}

Since the only pole of $\varphi$ in $\Omega$ is at $-15$, the function $\varphi$ is analytic on the unit disk. For $0<x<1$,
\begin{equation*}
\varphi(x)=\frac{A(x)}{D(x)}=\frac{\frac{2}{\pi}\mathcal{K}(\sqrt{x})}{\ln4-\frac{1}{2}\ln(1-x)}=\frac{2}{\pi}f_{\ln4}(x).
\end{equation*}
Hence, by \eqref{originalseries} and the uniqueness of Taylor coefficients,
\begin{equation} \label{varphiseries}
\varphi(z)=\sum_{n=0}^{\infty}a_nz^n,\quad |z|<1.
\end{equation}

Here and below, $A(t+i0)$, $D(t+i0)$ and $\varphi(t+i0)$ denote the limits from the upper half-plane. The next lemma computes these boundary values.

\begin{lemma}
Let $t>1$ and
\begin{equation} \label{notation}
r=\frac{1}{\sqrt t},\quad r'=\sqrt{1-r^2},\quad c=\ln\frac{4r}{r'}.
\end{equation}
Then
\begin{equation} \label{Aboundary} 
A(t+i0)=\frac{2r}{\pi}\left(\mathcal{K}(r)+i\mathcal{K}(r')\right)
\end{equation}
and
\begin{equation} \label{Dboundary} 
D(t+i0)=c+\frac{i\pi}{2}.
\end{equation}
Consequently,
\begin{equation} \label{Imvarphi} 
\operatorname{Im}\varphi(t+i0) =\frac{2r}{\pi}\frac{c\mathcal{K}(r')-\frac{\pi}{2}\mathcal{K}(r)}{c^2+\frac{\pi^2}{4}}.
\end{equation}
Moreover, for every compact interval $J\Subset(1,\infty)$,
\begin{equation} \label{Auniformboundary}
\sup_{t\in J}|A(t+i\eta)-A(t+i0)|\to0,\quad \eta\to0^+.
\end{equation}
\end{lemma}

\begin{proof}
Euler's integral formula for $A$ is
\begin{equation} \label{EulerA} 
A(z)=\frac{2}{\pi}\int_0^{\frac{\pi}{2}}(1-z\sin^2\theta)^{-\frac{1}{2}} \mathrm{d}\theta.
\end{equation}
For a fixed compact interval $J=[t_0,t_1]\Subset(1,\infty)$, denote that
\begin{equation*}
\theta_0(t)=\arcsin(t^{-\frac{1}{2}}), \quad h_t(\theta)=1-t\sin^2\theta.
\end{equation*}
Then
\begin{equation*}
\left|\frac{\partial h_t}{\partial\theta}(\theta_0(t))\right| =2\sqrt{t-1}\geq 2\sqrt{t_0-1}>0 
\end{equation*}
and the map $(t,\theta)\mapsto\partial_\theta h_t(\theta)$ is continuous. Since $J$ is compact and $t\mapsto\theta_0(t)$ is continuous, there exist $\delta_J>0$ and $c_J>0$ such that
\begin{equation*}
|\partial_\theta h_t(\theta)|\geq c_J
\end{equation*}
if $t\in J$ and $|\theta-\theta_0(t)|<\delta_J$. Since $h_t(\theta_0(t))=0$, the mean-value theorem gives
\begin{equation*}
|h_t(\theta)|\geq c_J|\theta-\theta_0(t)|.
\end{equation*}
Therefore, for $\eta>0$,
\begin{equation*}
\left|1-(t+i\eta)\sin^2\theta\right|^{-\frac{1}{2}}\leq c_J^{-\frac{1}{2}}|\theta-\theta_0(t)|^{-\frac{1}{2}}.
\end{equation*}
The same estimate holds for the upper boundary value obtained as $\eta\to0^+$. Hence
\begin{equation*}
\int_{|\theta-\theta_0(t)|<d}\left|1-(t+i\eta)\sin^2\theta\right|^{-\frac{1}{2}}\mathrm{d}\theta\leq4c_J^{-\frac{1}{2}}\sqrt{d}
\end{equation*}
for $0<d<\delta_J$, uniformly in $t\in J$ and $\eta\geq0$. On $|\theta-\theta_0(t)|\geq d$, the denominator is uniformly bounded away from zero, and the integrand converges uniformly to its upper boundary value as $\eta\to0^+$. Splitting the integral into these two regions and then letting $\eta\to0^+$ and $d\to0^+$ gives
\begin{equation*} 
\sup_{t\in J}|A(t+i\eta)-A(t+i0)|\to0,\quad \eta\to0^+.
\end{equation*}

If $0<\theta<\theta_0$, then $1-t\sin^2\theta>0$. If $\theta_0<\theta<\frac{\pi}{2}$, then
\begin{equation*}
1-(t+i0)\sin^2\theta=-(t\sin^2\theta-1)-i0.
\end{equation*} 
Thus \eqref{EulerA} gives
\begin{equation} \label{splitA} 
A(t+i0)=\frac{2}{\pi}\left[ \int_0^{\theta_0}\frac{\mathrm{d}\theta}{\sqrt{1-t\sin^2\theta}} +i\int_{\theta_0}^{\frac{\pi}{2}}\frac{\mathrm{d}\theta}{\sqrt{t\sin^2\theta-1}} \right].
\end{equation}
For the first integral in \eqref{splitA}, set $\sin\theta=r\sin u$. Since $\sin\theta_0=r$, the variable $u$ runs from $0$ to $\frac{\pi}{2}$, and
\begin{equation*}
\mathrm{d}\theta=\frac{r\cos u}{\sqrt{1-r^2\sin^2u}}\mathrm{d}u,\quad \sqrt{r^2-\sin^2\theta}=r\cos u.
\end{equation*}
Therefore
\begin{equation*}
r\int_0^{\theta_0}\frac{\mathrm{d}\theta}{\sqrt{r^2-\sin^2\theta}}=r\int_0^{\frac{\pi}{2}}\frac{\mathrm{d}u}{\sqrt{1-r^2\sin^2u}}=r\mathcal{K}(r).
\end{equation*} 
For the second integral in \eqref{splitA}, set $\cos\theta=r'\sin u$. Since $\cos\theta_0=r'$, the variable $u$ decreases from $\frac{\pi}{2}$ to $0$, and
\begin{equation*}
\mathrm{d}\theta=-\frac{r'\cos u}{\sqrt{1-(r')^2\sin^2u}}\mathrm{d}u,\quad \sqrt{\sin^2\theta-r^2}=r'\cos u.
\end{equation*}
Reversing the limits gives
\begin{equation*}
r\int_{\theta_0}^{\frac{\pi}{2}}\frac{\mathrm{d}\theta}{\sqrt{\sin^2\theta-r^2}}=r\int_0^{\frac{\pi}{2}}\frac{\mathrm{d}u}{\sqrt{1-(r')^2\sin^2u}}=r\mathcal{K}(r').
\end{equation*}
Together with the preceding identity, this yields \eqref{Aboundary}.

For the denominator, since
\begin{equation*}
1-(t+i0)=-(t-1)-i0,
\end{equation*}
we have
\begin{equation*}
\operatorname{Log}(1-t-i0)=\ln(t-1)-i\pi.
\end{equation*}
Moreover,
\begin{equation*}
t-1=\frac{1-r^2}{r^2}=\frac{(r')^2}{r^2}.
\end{equation*}
Therefore,
\begin{equation*}
D(t+i0)=\ln 4-\frac{1}{2}\ln(t-1)+\frac{i\pi}{2} =\ln\frac{4r}{r'}+\frac{i\pi}{2},
\end{equation*}
which is \eqref{Dboundary}. Finally,
\begin{equation*}
\begin{aligned}
\varphi(t+i0)=&\frac{2r}{\pi} \frac{\mathcal{K}(r)+i\mathcal{K}(r')}{c+i\frac{\pi}{2}} \\
=&\frac{2r}{\pi} \frac{\left(c\mathcal{K}(r)+\frac{\pi}{2}\mathcal{K}(r')\right) +i\left(c\mathcal{K}(r')-\frac{\pi}{2}\mathcal{K}(r)\right)} {c^2+\frac{\pi^2}{4}}.
\end{aligned}
\end{equation*}
Taking the imaginary part gives \eqref{Imvarphi}.
\end{proof}

\subsection{Positive density and a moment representation}

The sign in \eqref{Imvarphi} will follow from a real-variable inequality for $\frac{\mathcal{K}(r)}{\mathcal{K}(r')}$.  
Applying \cite[Corollary 1.9]{PonnusamyVuorinen1997}
with $a=b=\frac{1}{2}$ and $x=r^2$, we obtain 
\begin{equation*}
0<\mathcal{K}(r)-\ln\frac4{r'}<\frac{(r')^2}{r^2}\ln\frac{1}{r'}, \quad 0<r<1.
\end{equation*}
Since $r^{-2}$ remains bounded as $r\to1^-$, this yields 
\begin{equation} \label{Kasymp} 
\mathcal{K}(r)=\ln\frac4{r'}+O\left((r')^2\ln\frac{1}{r'}\right), \quad r\to1^-.
\end{equation}

\begin{lemma} \label{ellipticinequalitylemma}
For every $0<r<1$,
\begin{equation} \label{ellipticinequality}
\frac{\pi}{2}\frac{\mathcal{K}(r)}{\mathcal{K}(r')}>\ln\frac{4r}{r'}.
\end{equation}
\end{lemma}

\begin{proof}
By \cite[(19.9.5)]{NISTHandbook}, applied with $k=r'$,
\begin{equation*}
\frac{\pi}{2}\frac{\mathcal{K}(r)}{\mathcal{K}(r')}>\ln\frac{(1+\sqrt{r})^2}{r'}.
\end{equation*}
Since
\begin{equation*}
(1+\sqrt{r})^2-4r=(1-\sqrt{r})(1+3\sqrt{r})>0,
\end{equation*}
we have
\begin{equation*}
\ln\frac{(1+\sqrt{r})^2}{r'}>\ln\frac{4r}{r'}.
\end{equation*}
Thus \eqref{ellipticinequality} follows.
\end{proof}

Combining \eqref{Imvarphi} with \eqref{ellipticinequality} gives
\begin{equation*}
c\mathcal{K}(r')-\frac{\pi}{2}\mathcal{K}(r)<0,
\end{equation*}
and therefore
\begin{equation*}
\operatorname{Im}\varphi(t+i0)<0,\quad t>1.
\end{equation*}
Set
\begin{equation*}
\rho(t)=-\frac{1}{\pi}\operatorname{Im}\varphi(t+i0),\quad t>1.
\end{equation*}
Then $\rho(t)>0$ for every $t>1$.

\begin{lemma}
For every integer $n\geq 0$,
\begin{equation} \label{momentformula} 
a_n=\frac{(-1)^nC}{15^{n+1}}-M_n, \quad M_n=\int_1^\infty\frac{\rho(t)}{t^{n+1}} \mathrm{d}t>0.
\end{equation}
\end{lemma}

\begin{proof}
Fix an integer $n\geq0$ and set
\begin{equation*}
G_n(z)=\frac{\varphi(z)}{z^{n+1}}.
\end{equation*}
Let $R>16$, $0<\varepsilon<\frac{1}{2}$ and $0<\delta<\frac{\pi}{2}$. Let $\Gamma_{R,\varepsilon,\delta}$ be the positively oriented boundary of
\begin{equation*}
\left\{z\in\mathbb{C}:|z|<R,\ |z-1|>\varepsilon,\ \delta<\arg(z-1)<2\pi-\delta\right\},
\end{equation*}
where $\arg(z-1)\in(0,2\pi)$. Then $\Gamma_{R,\varepsilon,\delta}\subset\Omega$, and both $0$ and $-15$ lie in its interior. As $\delta\to0^+$, the two radial segments approach the upper and lower sides of the cut.

\textit{Step 1: Estimates on the contour boundary}. 
Let $0<\varepsilon<\frac{1}{2}$ and let $z\in\Omega$ satisfy $|z-1|=\varepsilon$. Write $z-1=\varepsilon\mathrm{e}^{i\gamma}$. Then for $0\leq\theta\leq\frac{\pi}{2}$,
\begin{equation*}
\left|1-z\sin^2\theta\right|^2=\left|\cos^2\theta-\varepsilon\mathrm{e}^{i\gamma}\sin^2\theta\right|^2\geq\left(\cos^2\theta-\varepsilon\sin^2\theta\right)^2.
\end{equation*}
Hence
\begin{equation*}
|A(z)|\leq \frac{2}{\pi}\int_0^{\frac{\pi}{2}} \frac{\mathrm{d}\theta}{\sqrt{\left|1-(1+\varepsilon)\sin^2\theta\right|}}.
\end{equation*}
Put
\begin{equation*}
r_\varepsilon=\frac{1}{\sqrt{1+\varepsilon}},\quad r_\varepsilon'=\sqrt{\frac{\varepsilon}{1+\varepsilon}}, \quad \theta_\varepsilon=\arcsin r_\varepsilon.
\end{equation*}
Since $1+\varepsilon=r_\varepsilon^{-2}$, the substitution $\sin\theta=r_\varepsilon\sin u$ gives
\begin{equation*}
\int_0^{\theta_\varepsilon}\frac{\mathrm{d}\theta}{\sqrt{1-(1+\varepsilon)\sin^2\theta}}=r_\varepsilon\mathcal{K}(r_\varepsilon).
\end{equation*}
For the second integral, the substitution $\cos\theta=r_\varepsilon'\sin u$ yields
\begin{equation*}
\int_{\theta_\varepsilon}^{\frac{\pi}{2}}\frac{\mathrm{d}\theta}{\sqrt{(1+\varepsilon)\sin^2\theta-1}}=r_\varepsilon\mathcal{K}(r_\varepsilon').
\end{equation*}
Consequently,
\begin{equation*}
|A(z)|\leq\frac{2r_\varepsilon}{\pi}\left(\mathcal{K}(r_\varepsilon)+\mathcal{K}(r_\varepsilon')\right).
\end{equation*}
By \eqref{Kasymp}, $\mathcal{K}(r_\varepsilon)=O(\ln\frac{1}{\varepsilon})$, while \eqref{originalseries} gives $\mathcal{K}(r_\varepsilon')=O(1)$. Therefore, it holds that
\begin{equation} \label{Anearonebound} 
|A(z)|=O\left(\ln\frac{1}{\varepsilon}\right)
\end{equation}
uniformly on $|z-1|=\varepsilon$ in $\Omega$. On the same circle,
\begin{equation*}
\operatorname{Re} D(z)=\ln4-\frac{1}{2}\ln\varepsilon,
\end{equation*}
so $|D(z)|\geq \ln4-\frac{1}{2}\ln\varepsilon$. Combining this with \eqref{Anearonebound} yields
\begin{equation} \label{varphinearone} 
\sup_{\substack{|z-1|=\varepsilon\\ z\in\Omega}}|\varphi(z)|=O(1), \quad \varepsilon\to0^+.
\end{equation}

Let $t\to1^+$ and use the notation of \eqref{notation}. Put $s=r'$. Then $s\to0^+$,
\begin{equation*}
t-1=\frac{s^2}{1-s^2},\quad \ln r=\frac{1}{2}\ln(1-s^2)=O(s^2).
\end{equation*}
With $L=\ln\frac{4}{s}$, formulas \eqref{Kasymp} and \eqref{originalseries} give
\begin{equation*}
\mathcal{K}(r)=L+O\left(s^2L\right),\quad \mathcal{K}(s)=\frac{\pi}{2}+O(s^2),\quad c=L+O(s^2).
\end{equation*}
Using these three expansions gives
\begin{equation*}
\begin{aligned}
c\mathcal{K}(s)-\frac{\pi}{2}\mathcal{K}(r)&=\left(L+O(s^2)\right)\left(\frac{\pi}{2}+O(s^2)\right) -\frac{\pi}{2}\left(L+O(s^2L)\right)\\
&=O(s^2L).
\end{aligned}
\end{equation*}
Moreover, $c=L+O(s^2)$ and $L\to\infty$, so $\frac cL\to1$. Hence, for all sufficiently small $s$, $|c|\geq \frac L2$ and therefore
\begin{equation*}
c^2+\frac{\pi^2}{4}\geq \frac{L^2}{4}.
\end{equation*}
Since $r\leq 1$, formula \eqref{Imvarphi} gives
\begin{equation*}
\rho(t) =O\left(\frac{s^2L}{L^2}\right) =O\left(\frac{s^2}{L}\right).
\end{equation*}
Finally,
\begin{equation*}
t-1=\frac{s^2}{1-s^2}\sim s^2,
\end{equation*}
and $L\geq 1$ for small $s$. Hence $\rho(t)=O(t-1)$ as $t\to1^+$. In particular, $\rho$ is integrable on $(1,1+\eta)$ for some $\eta>0$.

For the outer circle, let $|z|=R>1$ and write $z=R\mathrm{e}^{i\gamma}$. Since
\begin{equation*}
\left|1-R\mathrm{e}^{i\gamma}\sin^2\theta\right|^2\geq\left(1-R\sin^2\theta\right)^2,
\end{equation*}
Euler's integral representation gives
\begin{equation*}
|A(z)|\leq\frac{2}{\pi}\int_0^{\frac{\pi}{2}}\frac{\mathrm{d}\theta}{\sqrt{|1-R\sin^2\theta|}}.
\end{equation*}
Set
\begin{equation*}
r=R^{-\frac{1}{2}},\quad \theta_0=\arcsin r.
\end{equation*}
Then
\begin{equation*}
\int_0^{\frac{\pi}{2}}\frac{\mathrm{d}\theta}{\sqrt{|1-R\sin^2\theta|}}=\int_0^{\theta_0}\frac{\mathrm{d}\theta}{\sqrt{1-R\sin^2\theta}}+\int_{\theta_0}^{\frac{\pi}{2}}\frac{\mathrm{d}\theta}{\sqrt{R\sin^2\theta-1}}.
\end{equation*}
Since $R=\frac{1}{r^2}$, the substitutions $\sin\theta=r\sin u$ and $\cos\theta=r'\sin u$ give
\begin{equation*}
\int_0^{\theta_0}\frac{\mathrm{d}\theta}{\sqrt{1-R\sin^2\theta}}=r\mathcal{K}(r)
\end{equation*}
and
\begin{equation*}
\int_{\theta_0}^{\frac{\pi}{2}}\frac{\mathrm{d}\theta}{\sqrt{R\sin^2\theta-1}}=r\mathcal{K}(r').
\end{equation*}
Therefore,
\begin{equation*}
|A(z)|\leq\frac{2r}{\pi}\left(\mathcal{K}(r)+\mathcal{K}(r')\right).
\end{equation*}
As $R\to\infty$, we have $r=R^{-\frac{1}{2}}\to0$. The power series \eqref{originalseries} of $\mathcal{K}$ gives
\begin{equation*}
\mathcal{K}(r)=O(1),
\end{equation*}
while \eqref{Kasymp} implies
\begin{equation*}
\mathcal{K}(r')=O\left(\ln\frac{1}{r}\right)=O(\ln R).
\end{equation*}
Hence
\begin{equation*}
|A(z)|=O\left(R^{-\frac{1}{2}}\ln R\right)
\end{equation*}
uniformly for $|z|=R$ with $z\in\Omega$.

Moreover,
\begin{equation*}
\operatorname{Re}D(z)=\ln4-\frac{1}{2}\ln|1-z|.
\end{equation*}
Since $|1-z|\geq R-1$, for all sufficiently large $R$,
\begin{equation*}
|D(z)|\geq-\operatorname{Re}D(z)\geq\frac{1}{2}\ln(R-1)-\ln4\geq\frac{1}{4}\ln R.
\end{equation*}
Combining the preceding estimates yields
\begin{equation} \label{largevarphi}
\sup_{\substack{|z|=R\\ z\in\Omega}}|\varphi(z)|=O\left(R^{-\frac{1}{2}}\right).
\end{equation}

On the upper side of the cut, let $r=t^{-\frac{1}{2}}\to0$ and set $L=\ln\frac{1}{r}$. The power series \eqref{originalseries} of $\mathcal{K}$ and \eqref{Kasymp} applied to $\mathcal{K}(r')$, give
\begin{equation*}
\mathcal{K}(r)=O(1) \quad \mathcal{K}(r')=O(L).
\end{equation*}
Since $r'=\sqrt{1-r^2}=1+O(r^2)$, we have $\ln r'=O(r^2)$ and hence
\begin{equation*}
c=\ln\frac{4r}{r'}=-L+\ln4+O(r^2).
\end{equation*}
Therefore, for all sufficiently small $r$,
\begin{equation*}
\frac{1}{2}L\leq |c|\leq2L.
\end{equation*}
Then we obtain
\begin{equation*}
\left|c\mathcal{K}(r')-\frac{\pi}{2}\mathcal{K}(r)\right|\leq |c|\mathcal{K}(r')+\frac{\pi}{2}\mathcal{K}(r)=O(L^2).
\end{equation*}
Moreover,
\begin{equation*}
c^2+\frac{\pi^2}{4}\geq c^2\geq\frac{1}{4}L^2.
\end{equation*}
It follows from \eqref{Imvarphi} that
\begin{equation*}
|\operatorname{Im}\varphi(t+i0)|=O(r).
\end{equation*}
Since $r=t^{-\frac{1}{2}}$,
\begin{equation*}
\rho(t)=O\left(t^{-\frac{1}{2}}\right),\quad t\to\infty.
\end{equation*}
Together with $\rho(t)=O(t-1)$ as $t\to1^+$, this yields
\begin{equation*}
\int_1^\infty\frac{\rho(t)}{t^{n+1}}\mathrm{d}t<\infty,\quad n\geq0,
\end{equation*}
Thus $M_n$ in \eqref{momentformula} is finite for every $n\geq0$.

\vspace{2mm}

\noindent\textit{Step 2: Residue evaluation and passage to the branch cut}. 
Inside $\Gamma_{R,\varepsilon,\delta}$, the function $G_n$ has precisely two poles. Since
\begin{equation*}
G_n(z)=\sum_{k=0}^{\infty}a_kz^{k-n-1}
\end{equation*}
near $z=0$, we have
\begin{equation*}
\operatorname*{Res}_{z=0}G_n(z)=a_n.
\end{equation*}
At $z=-15$, Lemma \ref{polelemma} gives
\begin{equation*}
\operatorname*{Res}_{z=-15}G_n(z)=\frac{C}{(-15)^{n+1}}=\frac{(-1)^{n+1}C}{15^{n+1}}.
\end{equation*}
Hence the residue theorem yields
\begin{equation} \label{residueidentity}
\frac{1}{2\pi i}\int_{\Gamma_{R,\varepsilon,\delta}}G_n(z)\mathrm{d}z=a_n+\frac{(-1)^{n+1}C}{15^{n+1}}.
\end{equation}

Let $s_R(\delta)>0$ be determined by
\begin{equation*}
\left|1+s_R(\delta)\mathrm{e}^{i\delta}\right|=R.
\end{equation*}
Since
\begin{equation*}
\left|1+s\mathrm{e}^{i\delta}\right|^2=1+s^2+2s\cos\delta,
\end{equation*}
we obtain
\begin{equation*}
s_R(\delta)=-\cos\delta+\sqrt{R^2-\sin^2\delta},
\end{equation*}
and therefore
$
s_R(\delta)\to R-1
$
as $\delta\to0^+$.
The integrals over the upper and lower radial segments are
\begin{equation*}
I_+(\delta)=\int_{\varepsilon}^{s_R(\delta)}G_n\left(1+s\mathrm{e}^{i\delta}\right)\mathrm{e}^{i\delta}\mathrm{d}s
\end{equation*}
and
\begin{equation*}
I_-(\delta)=\int_{s_R(\delta)}^{\varepsilon}G_n\left(1+s\mathrm{e}^{-i\delta}\right)\mathrm{e}^{-i\delta}\mathrm{d}s.
\end{equation*} 

Choose $\delta_0>0$ sufficiently small that
\begin{equation*}
\frac{\varepsilon}{2}\leq s_R(\delta)\leq R
\end{equation*}
for $0<\delta<\delta_0$. Set
\begin{equation*}
J_{\varepsilon,R}:=\left[1+\frac{\varepsilon}{4},1+R\right]\Subset(1,\infty).
\end{equation*}
For $\frac{\varepsilon}{2}\leq s\leq R$, write
\begin{equation*}
t_\delta(s)=1+s\cos\delta,\quad \eta_\delta(s)=s\sin\delta.
\end{equation*}
Then $t_\delta(s)\in J_{\varepsilon,R}$ and
\begin{equation*}
0\leq\eta_\delta(s)\leq R\sin\delta\to0
\end{equation*}
uniformly in $s$. Hence \eqref{Auniformboundary} gives
\begin{equation*}
\sup_{\frac{\varepsilon}{2}\leq s\leq R}\left|A\left(t_\delta(s)+i\eta_\delta(s)\right)-A\left(t_\delta(s)+i0\right)\right|\to0.
\end{equation*}
Since $t\mapsto A(t+i0)$ is continuous on the compact interval $J_{\varepsilon,R}$, it is uniformly continuous there. Therefore
\begin{equation*}
\sup_{\frac{\varepsilon}{2}\leq s\leq R}\left|A\left(t_\delta(s)+i0\right)-A(1+s+i0)\right|\to0.
\end{equation*}
Combining the preceding two estimates yields
\begin{equation*}
\sup_{\frac{\varepsilon}{2}\leq s\leq R}\left|A\left(1+s\mathrm{e}^{i\delta}\right)-A(1+s+i0)\right|\to0.
\end{equation*} 
For $s>0$ and $0<\delta<\pi$, the principal logarithm gives
\begin{equation*}
\operatorname{Log}\left(-s\mathrm{e}^{i\delta}\right)=\ln s+i(\delta-\pi),
\end{equation*}
and hence
\begin{equation*}
D\left(1+s\mathrm{e}^{i\delta}\right)=\ln4-\frac{1}{2}\ln s+\frac{i}{2}(\pi-\delta).
\end{equation*}
Thus
\begin{equation*}
D\left(1+s\mathrm{e}^{i\delta}\right)\to D(1+s+i0)=\ln4-\frac{1}{2}\ln s+\frac{i\pi}{2}
\end{equation*}
uniformly for $\frac{\varepsilon}{2}\leq s\leq R$. Moreover, for all sufficiently small $\delta$,
\begin{equation*}
\left|D\left(1+s\mathrm{e}^{i\delta}\right)\right|\geq\frac{\pi-\delta}{2}\geq\frac{\pi}{4}.
\end{equation*}
Combining these estimates with
\begin{equation*}
1+s\mathrm{e}^{i\delta}\to1+s, \quad \mathrm{e}^{i\delta}\to1
\end{equation*}
uniformly in $s$, we obtain
\begin{equation*}
\sup_{\frac{\varepsilon}{2}\leq s\leq R}\left|G_n\left(1+s\mathrm{e}^{i\delta}\right)\mathrm{e}^{i\delta}-G_n(1+s+i0)\right|\to0.
\end{equation*}
Together with $s_R(\delta)\to R-1$, this gives
\begin{equation*}
\lim_{\delta\to0^+}I_+(\delta)=\int_{\varepsilon}^{R-1}G_n(1+s+i0)\mathrm{d}s=\int_{1+\varepsilon}^{R}G_n(t+i0)\mathrm{d}t.
\end{equation*} 
Since $G_n(\overline{z})=\overline{G_n(z)}$, we have $I_-(\delta)=-\overline{I_+(\delta)}$. Hence
\begin{equation*}
\lim_{\delta\to0^+}I_-(\delta)=-\int_{1+\varepsilon}^{R}G_n(t-i0)\mathrm{d}t.
\end{equation*}

For fixed $\varepsilon$, the boundary values of $G_n$ on $|z-1|=\varepsilon$ are bounded. Thus
\begin{equation*}
C_{\varepsilon,n}:=\sup_{\substack{|z-1|=\varepsilon\\ z\in\Omega}}|G_n(z)|<\infty.
\end{equation*}
The omitted arc on $|z-1|=\varepsilon$ has length $2\varepsilon\delta$, and its contribution is bounded by $2\varepsilon\delta C_{\varepsilon,n}\to0$ as $\delta\to0^+$. 
For the outer circle, let $\ell_R(\delta)$ be the length of the omitted arc. Its endpoints converge to $R$ as $\delta\to0^+$, so $\ell_R(\delta)\to0$. The boundary values of $G_n$ on $|z|=R$ are also bounded, and hence
\begin{equation*}
C_{R,n}:=\sup_{\substack{|z|=R\\ z\in\Omega}}|G_n(z)|<\infty.
\end{equation*}
Therefore, the contribution of the omitted outer arc is bounded by $C_{R,n}\ell_R(\delta)\to0$ as $\delta\to0^+$. 

\vspace{2mm}

\noindent\textit{Step 3: Combining the cut contributions and taking the limits}.
By Step 2,
\begin{equation*}
\lim_{\delta\to0^+}\left(I_+(\delta)+I_-(\delta)\right)=\int_{1+\varepsilon}^{R}\left(G_n(t+i0)-G_n(t-i0)\right)\mathrm{d}t.
\end{equation*}
Since $\varphi(t-i0)=\overline{\varphi(t+i0)}$, we have
\begin{equation*}
G_n(t+i0)-G_n(t-i0)=\frac{2i\operatorname{Im}\varphi(t+i0)}{t^{n+1}}=-\frac{2\pi i\rho(t)}{t^{n+1}}.
\end{equation*}
Hence
\begin{equation} \label{finitebankjump}
\lim_{\delta\to0^+}\left(I_+(\delta)+I_-(\delta)\right)=-2\pi i\int_{1+\varepsilon}^{R}\frac{\rho(t)}{t^{n+1}}\mathrm{d}t.
\end{equation}

Since $|z|\geq1-\varepsilon\geq\frac{1}{2}$ on $|z-1|=\varepsilon$, \eqref{varphinearone} gives
\begin{equation*}
\sup_{\substack{|z-1|=\varepsilon\\ z\in\Omega}}|G_n(z)|=O(1).
\end{equation*}
The inner circle has length $2\pi\varepsilon$, so its contribution is $O(\varepsilon)$ and tends to zero. Moreover, since $\rho(t)=O(t-1)$ as $t\to1^+$,
\begin{equation*}
\frac{\rho(t)}{t^{n+1}}=O(t-1),\quad t\to1^+.
\end{equation*}
Hence the integral in \eqref{finitebankjump} converges at $t=1$.
By \eqref{largevarphi},
\begin{equation*}
\sup_{\substack{|z|=R\\ z\in\Omega}}|G_n(z)|=O\left(R^{-n-\frac{3}{2}}\right).
\end{equation*}
Since the outer circle has length $2\pi R$, its contribution is $O(R^{-n-\frac{1}{2}})$ and tends to zero. Also, since $\rho(t)=O\left(t^{-\frac{1}{2}}\right)$ as $t\to\infty$,
\begin{equation*}
\frac{\rho(t)}{t^{n+1}}=O\left(t^{-n-\frac{3}{2}}\right).
\end{equation*}
Thus the integral in \eqref{finitebankjump} converges at infinity. 
Consequently, after letting $\delta\to0^+$, then $\varepsilon\to0^+$, and finally $R\to\infty$,  
\begin{equation*}
-2\pi i\int_1^\infty\frac{\rho(t)}{t^{n+1}}\mathrm{d}t=-2\pi iM_n.
\end{equation*}
Passing to the limits in \eqref{residueidentity} gives 
\begin{equation*}
a_n=\frac{(-1)^nC}{15^{n+1}}-M_n.
\end{equation*}
Since $\rho(t)>0$ for every $t>1$,
\begin{equation*}
M_n=\int_1^\infty\frac{\rho(t)}{t^{n+1}}\mathrm{d}t>0.
\end{equation*}
This proves \eqref{momentformula}.
\end{proof}

\subsection{Moment comparison and coefficient signs}

The representation \eqref{momentformula} gives $a_n<0$ immediately for odd $n$. For even $n\geq4$, it remains to prove that the moment term exceeds the pole contribution. This will follow from the strict log-convexity of the moments together with the initial coefficient estimates.

\begin{lemma}
The moments in \eqref{momentformula} are strictly log-convex, that is,
\begin{equation*}
M_{n+1}^2<M_nM_{n+2},\quad n\geq 0.
\end{equation*}
Equivalently,
\begin{equation} \label{ratioincrease} 
\frac{M_{n+1}}{M_n}<\frac{M_{n+2}}{M_{n+1}},\quad n\geq 0.
\end{equation}
\end{lemma}

\begin{proof}
Applying Cauchy--Schwarz inequality with the positive measure $\rho(t) \mathrm{d}t$, we have
\begin{equation*}
\left(\int_1^\infty\frac{\rho(t)}{t^{n+2}} \mathrm{d}t\right)^2 \leq \left(\int_1^\infty\frac{\rho(t)}{t^{n+1}} \mathrm{d}t\right) \left(\int_1^\infty\frac{\rho(t)}{t^{n+3}} \mathrm{d}t\right).
\end{equation*}
Thus $M_{n+1}^2\leq M_nM_{n+2}$. Equality would imply
\begin{equation*}
t^{-(n+1)/2}=\lambda t^{-(n+3)/2}
\end{equation*}
for some constant $\lambda$ and for $\rho(t) \mathrm{d}t$-almost every $t$. Hence $t=\lambda$ almost everywhere, which is impossible because $\rho(t)>0$ on all of $(1,\infty)$. Therefore $M_{n+1}^2<M_nM_{n+2}$. Division by $M_nM_{n+1}>0$ gives \eqref{ratioincrease}.
\end{proof}

\begin{lemma} \label{momentbarrier}
For $p=\ln 4$, the moments in \eqref{momentformula} satisfy
\begin{equation} \label{Mnbarrier} 
M_n>\frac{C}{15^{n+1}},\quad n\geq 3.
\end{equation}
\end{lemma}

\begin{proof}
We first record
\begin{equation} \label{logbounds}
\frac{4}{3}<p<\frac{7}{5}.
\end{equation}
For the lower bound,
\begin{equation*}
\ln2=\int_1^{\frac{3}{2}}\left(\frac{1}{t}+\frac{1}{3-t}\right)\mathrm{d}t.
\end{equation*}
Since $t(3-t)\leq\frac{9}{4}$ on $\left[1,\frac{3}{2}\right]$,
\begin{equation*}
\frac{1}{t}+\frac{1}{3-t}=\frac{3}{t(3-t)}\geq\frac{4}{3},
\end{equation*}
with strict inequality except at $t=\frac{3}{2}$. Hence $\ln2>\frac{2}{3}$ and therefore $p=2\ln2>\frac{4}{3}$. For the upper bound,
\begin{equation*}
\mathrm{e}^{\frac{7}{10}}>1+\frac{7}{10}+\frac{1}{2}\left(\frac{7}{10}\right)^2+\frac{1}{6}\left(\frac{7}{10}\right)^3=\frac{12013}{6000}>2,
\end{equation*}
so $\ln2<\frac{7}{10}$ and $p<\frac{7}{5}$.

\vspace{2mm}

By \eqref{coeffs} and \eqref{logbounds},
\begin{equation} \label{a2positive}
a_2=\frac{(3p-4)^2}{64p^3}>0.
\end{equation}
For $a_3$, let
\begin{equation*}
F(p)=75p^3-230p^2+240p-96.
\end{equation*}
Since
\begin{equation*}
F'(p)=5(45p^2-92p+48)>0,
\end{equation*}
because $92^2-4\cdot45\cdot48=-176<0$, the function $F$ is strictly increasing. Hence, by $p<\frac{7}{5}$,
\begin{equation*}
F(p)<F\left(\frac{7}{5}\right)=-5.
\end{equation*}
Therefore
\begin{equation} \label{a3estimate}
-a_3>\frac{5}{768\left(\frac{7}{5}\right)^4}>\frac{1}{768}>\frac{64}{15^4},
\end{equation}
where $\left(\frac{7}{5}\right)^4<5$ and $64\cdot768<15^4$.

Taking $n=2$ in \eqref{momentformula} and using \eqref{a2positive}, we obtain
\begin{equation*}
M_2<\frac{C}{15^3}.
\end{equation*}
For $n=3$, \eqref{momentformula} gives
\begin{equation*}
M_3=-a_3-\frac{C}{15^4}.
\end{equation*}
Since $C<32$ by \eqref{Cbound}, \eqref{a3estimate} yields
\begin{equation*}
-a_3>\frac{64}{15^4}>\frac{2C}{15^4},
\end{equation*}
and hence
\begin{equation} \label{M3lower}
M_3>\frac{C}{15^4}.
\end{equation}
Since $M_2>0$, the two preceding estimates give
\begin{equation*}
\frac{M_3}{M_2}>\frac{C/15^4}{C/15^3}=\frac{1}{15}.
\end{equation*}
By \eqref{ratioincrease}, 
\begin{equation} \label{allratios}
\frac{M_{n+1}}{M_n}>\frac{1}{15},\quad n\geq2.
\end{equation}

We now prove \eqref{Mnbarrier} by induction. The case $n=3$ is \eqref{M3lower}. If
$
M_n>\frac{C}{15^{n+1}}
$
for some $n\geq3$, then \eqref{allratios} gives
\begin{equation*}
M_{n+1}>\frac{M_n}{15}>\frac{C}{15^{n+2}}.
\end{equation*}
Thus \eqref{Mnbarrier} holds for every $n\geq3$.
\end{proof}

\begin{proof}[Proof of the first assertion of Theorem \ref{main}]
Let $n\geq3$. If $n$ is odd, \eqref{momentformula} gives
\begin{equation*}
a_n=-\frac{C}{15^{n+1}}-M_n<0.
\end{equation*}
If $n$ is even, then $n\geq4$, and Lemma \ref{momentbarrier} gives
\begin{equation*}
a_n=\frac{C}{15^{n+1}}-M_n<0.
\end{equation*}
Hence
\begin{equation} \label{allnegative}
a_n<0,\quad n\geq3.
\end{equation}

Since $\varphi$ is analytic for $|z|<1$ and $f_{\ln4}=\frac{\pi}{2}\varphi$, the expansion \eqref{varphiseries} may be differentiated termwise. Thus
\begin{equation*}
-f_{\ln4}^{\prime\prime\prime}(x)=\frac{\pi}{2}\sum_{n=3}^{\infty}(-a_n)n(n-1)(n-2)x^{n-3},\quad |x|<1.
\end{equation*}
For every integer $m\geq0$, 
\begin{equation*}
\left(-f_{\ln4}^{\prime\prime\prime}\right)^{(m)}(x)=\frac{\pi}{2}\sum_{n=m+3}^{\infty}(-a_n)\frac{n!}{(n-m-3)!}x^{n-m-3},\quad 0<x<1.
\end{equation*}
By \eqref{allnegative}, every term in the series is positive for $0<x<1$. Therefore
\begin{equation*}
\left(-f_{\ln4}^{\prime\prime\prime}\right)^{(m)}(x)>0,\quad m\geq0,\quad 0<x<1.
\end{equation*}
This proves the first assertion of Theorem \ref{main} and hence Conjecture 1.
\end{proof}

\section{Proof of Conjecture 2} \label{sec-second}

Set $p=q=\ln4$ and write $a_n=a_n(\ln4)$ and $b_n=b_n(\ln4)$. By \eqref{originalseries},
\begin{equation*}
\frac{2}{\pi}f_{\ln4}(x)=\sum_{n=0}^{\infty}a_nx^n,\quad \frac{\pi}{2}g_{\ln4}(x)=\sum_{n=0}^{\infty}b_nx^n
\end{equation*}
for $x$ sufficiently close to $0$. Since $g_{\ln4}=1/f_{\ln4}$,
\begin{equation} \label{inversecoeffidentity}
\left(\sum_{n=0}^{\infty}a_nx^n\right)\left(\sum_{n=0}^{\infty}b_nx^n\right)=1.
\end{equation}

Set
\begin{equation*}
P(x)=a_0+a_1x+a_2x^2,\quad H(x)=\sum_{n=3}^{\infty}(-a_n)x^n.
\end{equation*}
By \eqref{allnegative}, every coefficient of $H$ is strictly positive, and
\begin{equation*} 
\varphi(x)=P(x)-H(x).
\end{equation*}

We first determine the Taylor coefficients of $\frac{1}{P(x)}$.

\begin{lemma} \label{Pinverselemma}
The Taylor series of $\frac{1}{P(x)}$ at the origin has strictly positive coefficients.
\end{lemma}

\begin{proof}
By \eqref{coeffs},
\begin{equation*}
P(x)=\frac{1}{p}\left(1-\alpha x+\beta x^2\right),
\end{equation*}
where
\begin{equation*}
\alpha=\frac{2-p}{4p},\quad \beta=\frac{(3p-4)^2}{64p^2}.
\end{equation*}
By \eqref{logbounds}, $\frac{4}{3}<p<\frac{7}{5}<\frac{3}{2}$, so $\alpha>0$ and $\beta>0$. Moreover,
\begin{equation} \label{discriminant}
\alpha^2-4\beta=\frac{(2-p)^2-(3p-4)^2}{16p^2}=\frac{(p-1)(3-2p)}{4p^2}>0.
\end{equation}
Thus the equation
$
s^2-\alpha s+\beta=0
$
has two distinct real roots $\lambda$ and $\mu$. By Vieta's formulas,
\begin{equation*}
\lambda+\mu=\alpha>0,\quad \lambda\mu=\beta>0,
\end{equation*}
so $\lambda>0$ and $\mu>0$. Therefore
\begin{equation*}
P(x)=\frac{1}{p}(1-\lambda x)(1-\mu x).
\end{equation*}
Since $P(0)=\frac{1}{p}\neq0$, the function $\frac{1}{P(x)}$ is analytic in a neighborhood of the origin. For
$
|x|<\min\{\frac{1}{\lambda},\frac{1}{\mu}\},
$
the two geometric series converge absolutely, and hence
\begin{equation*} 
\frac{1}{P(x)}=p\left(\sum_{j=0}^{\infty}\lambda^jx^j\right)\left(\sum_{k=0}^{\infty}\mu^kx^k\right).
\end{equation*}
By the Cauchy product formula, the coefficient of $x^n$ is
$
p\sum_{j=0}^{n}\lambda^j\mu^{n-j}>0.
$
\end{proof}

\begin{lemma} \label{Anonzero}
The function $A$ has no zeros in the unit disk. Consequently, $\frac{1}{\varphi}=\frac{D}{A}$ is analytic for $|z|<1$.
\end{lemma}

\begin{proof}
Let $|z|<1$. For $0\leq\theta\leq\frac{\pi}{2}$,
\begin{equation*}
\operatorname{Re}\left(1-z\sin^2\theta\right)\geq1-|z|\sin^2\theta\geq1-|z|>0.
\end{equation*}
Thus $w=1-z\sin^2\theta$ lies in the open right half-plane. Then we have $|\arg w|<\frac{\pi}{2}$ and hence
\begin{equation*}
\operatorname{Re}\left(w^{-\frac{1}{2}}\right)
=|w|^{-\frac{1}{2}}\cos\left(\frac{\arg w}{2}\right)>0.
\end{equation*} 
By \eqref{EulerA},
\begin{equation*}
\operatorname{Re}A(z)=\frac{2}{\pi}\int_0^{\frac{\pi}{2}}\operatorname{Re}\left(1-z\sin^2\theta\right)^{-\frac{1}{2}}\mathrm{d}\theta>0.
\end{equation*}
Hence $A(z)\neq0$ for $|z|<1$. Since $D$ is analytic in the unit disk, the quotient $\frac{D}{A}$ is analytic there. Since $\varphi=A/D$, we have
\begin{equation*}
\frac{1}{\varphi(z)}=\frac{D(z)}{A(z)},\quad |z|<1.
\end{equation*}
\end{proof}

We now prove the second assertion of Theorem \ref{main}.

\begin{proof}[Proof of the second assertion of Theorem \ref{main}]
By Lemma \ref{Pinverselemma}, choose $0<r_0<1$ such that
\begin{equation*}
\frac{1}{P(x)}=\sum_{n=0}^{\infty}c_nx^n,\quad c_n>0,
\end{equation*}
for $|x|<r_0$. Since
$
H(x)=\sum_{n=3}^{\infty}(-a_n)x^n
$
and $-a_n>0$ for $n\geq3$, the Cauchy product gives
\begin{equation*}
\frac{H(x)}{P(x)}=\sum_{n=3}^{\infty}d_nx^n,
\end{equation*}
where
\begin{equation*}
d_n=\sum_{k=3}^{n}(-a_k)c_{n-k}>0,\quad n\geq3.
\end{equation*}
Set
\begin{equation*}
Q(x):=\frac{H(x)}{P(x)}=\sum_{n=3}^{\infty}d_nx^n,\quad d_n>0.
\end{equation*}
Since $Q(0)=0$, after decreasing $r_0$ if necessary, we have $|Q(x)|<1$ for $|x|<r_0$. Hence
\begin{equation*}
\frac{1}{\varphi(x)}=\frac{1}{P(x)}\frac{1}{1-Q(x)}
=\frac{1}{P(x)}\sum_{m=0}^{\infty}Q(x)^m.
\end{equation*}
Since $Q^m$ starts at degree $3m$, only finitely many terms contribute to each Taylor coefficient of the last series. Thus every Taylor coefficient of $\sum_{m=0}^{\infty}Q^m$ is nonnegative. Since every Taylor coefficient of $\frac{1}{P}$ is strictly positive by Lemma \ref{Pinverselemma}, every Taylor coefficient of $\frac{1}{\varphi}$ is strictly positive. By \eqref{inversecoeffidentity}, these coefficients are $b_n$, and hence
\begin{equation} \label{allbpositive}
b_n>0,\quad n\geq0.
\end{equation} 
By Lemma \ref{Anonzero}, $\frac{1}{\varphi}=\frac{D}{A}$ is analytic for $|z|<1$. Since
\begin{equation*}
g_{\ln4}(x)=\frac{2}{\pi}\frac{1}{\varphi(x)},
\end{equation*}
termwise differentiation gives
\begin{equation*}
g_{\ln4}^{(m)}(x)=\frac{2}{\pi}\sum_{n=m}^{\infty}b_n\frac{n!}{(n-m)!}x^{n-m},\quad m\geq0,\quad 0<x<1.
\end{equation*}
By \eqref{allbpositive}, 
\begin{equation*}
g_{\ln4}^{(m)}(x)>0,\quad m\geq0,\quad 0<x<1.
\end{equation*}
This proves the second assertion of Theorem \ref{main} and hence Conjecture 2.
\end{proof}

\end{document}